\documentclass[preprint,review,10pt]{elsarticle}

\usepackage{amsmath}
\usepackage{color}
\usepackage{verbatim}
\usepackage{amsfonts}

\usepackage{graphicx}
\usepackage{grffile}
\usepackage{caption}
\usepackage{subcaption}

\newcommand{\rset}{\mathbb{R}}
\newcommand{\bo}{\mathbf}
\newtheorem{lema}{Lemma}

\newtheorem{theorem}{Theorem}
\newproof{proof}{Proof}
\newcommand{\neigh}{\mathcal{N}^{i}}
\newcommand{\Jvec}{J^{\neigh}}

\begin{document}
\begin{frontmatter}

\title{Efficient parallel coordinate descent algorithm for convex
optimization problems with separable constraints: application to
distributed MPC\tnoteref{acknow}}

\tnotetext[acknow]{ The research leading to these results has
received funding from: the European Union, Seventh Framework
Programme (FP7/2007--2013) under grant agreement no 248940;
CNCSIS-UEFISCSU (project TE, no. 19/11.08.2010); ANCS (project PN
II, no. 80EU/2010); Sectoral Operational Programme Human Resources
Development 2007-2013 of the Romanian Ministry of Labor, Family and
Social Protection through the Financial Agreement
POSDRU/89/1.5/S/62557.}

\author[First]{Ion Necoara}
\author[First]{Dragos Clipici}

\address[First]{University Politehnica Bucharest, Automatic Control and Systems Engineering
Department, 060042 Bucharest, Romania \\ (e-mail:
ion.necoara@acse.pub.ro, d.clipici@acse.pub.ro)}

\begin{abstract}
In this paper we propose a parallel coordinate descent algorithm for
solving smooth convex optimization problems with separable
constraints that may arise e.g. in distributed model predictive
control (MPC)  for linear network systems. Our algorithm is based on
block coordinate descent updates in parallel and has a very simple
iteration. We prove (sub)linear rate of convergence for the new
algorithm under standard assumptions for smooth convex optimization.
Further,  our algorithm uses  local information and thus is suitable
for distributed implementations. Moreover,  it has low iteration
complexity, which makes it appropriate for embedded control. An MPC
scheme based on this new parallel algorithm is derived, for which
every subsystem in the network can compute feasible and stabilizing
control inputs using distributed and cheap computations. For
ensuring stability of the MPC scheme, we use a terminal cost
formulation derived from a distributed synthesis. Preliminary
numerical tests show better performance for our optimization
algorithm  than other existing methods.
\end{abstract}

\begin{keyword}
Coordinate descent optimization, parallel algorithm, (sub)linear
convergence rate, distributed model predictive control, embedded
control.
\end{keyword}

\end{frontmatter}

\section{Introduction}
\label{secintr}

Model predictive control (MPC) has become a popular advanced control
technology implemented in network systems due to its ability to
handle hard input and state constraints \cite{RawMay:09}. Network
systems are usually modeled  by a graph  whose nodes represent
subsystems and whose arcs indicate dynamic couplings. These types of
systems are complex and large in dimension, whose structures may be
hierarchical and they have multiple decision-makers (e.g. process
control \cite{Sca:09}, traffic and power systems
\cite{Godbole,Venkat}, flight formation~\cite{Massioni}).

Decomposition methods represent a very powerful tool for solving
distributed MPC  problems in network systems. The basic idea of
these methods is to decompose the original large optimization
problem into smaller subproblems. Decomposition methods can be
divided in two main classes: primal  and dual decomposition methods.
In primal decomposition the optimization problem is solved using the
original formulation and variables  via methods such as
interior-point, feasible directions, Gauss-Jacobi type and others
\cite{CamSch:11,Dunbar,FarSca:12,Stewart,Venkat}. In dual
decomposition the original problem is rewritten using Lagrangian
relaxation for the coupling constraints and the dual problem is
solved with a Newton or (sub)gradient algorithm
\cite{AlvLim:11,BerTsi:89,DoaKev:11,Nec:08,NecNed:11}. In
\cite{Stewart,Venkat} cooperative based distributed MPC algorithms
are proposed based on Gauss-Jacobi iterations, where asymptotic
convergence to the centralized solution  and  feasibility for their
iterates is proved. In \cite{Dunbar,FarSca:12} non-cooperative
algorithms are derived for distributed MPC problems, where
communication takes place only between neighbors. In
\cite{CamSch:11} a distributed algorithm based on interior-point
methods  is proposed whose iterates converge to the centralized
solution. In \cite{AlvLim:11,DoaKev:11,Nec:08,NecNed:11} dual
distributed gradient algorithms based on Lagrange relaxation of the
coupling constraints are presented for solving MPC problems,
algorithms which usually produce feasible and optimal primal
solutions in the limit. While much research has focused on a  dual
approach, our work develops  a primal method that ensures constraint
feasibility, has low iteration complexity and provides estimates on
suboptimality.

Further, MPC schemes tend to be quite costly computation-wise
compared with classical control methods, e.g.  PID controllers,  so
that for these advanced schemes we need  hardware with a reasonable
amount of computational power that is  embedded on the subsystems.
Therefore, research for distributed and embedded MPC has gained
momentum in the past few years. The concept behind embedded MPC is
designing a control scheme that can be implemented on autonomous
electronic hardware, e.g  programmable logic controllers (PLC)
\cite{ValRos:11} or  field-programmable gate arrays (FPGAs)
\cite{KerriganFPGA}. Such devices vary widely in both computational
power and memory storage capabilities as well as cost. As a result,
there has been a growing focus on making MPC schemes faster by
reducing problem size  and improving the computational efficiency
through decentralization \cite{Sca:09},  moving block strategies
(e.g. by using latent variables \cite{GolMac:10} or Laguerre
functions \cite{Wan:04}) and other procedures, allowing these
schemes to be implemented on cheaper hardware with little computational power.

The main contribution of this paper is the development of a parallel
coordinate descent algorithm for smooth convex optimization problems
with separable constraints that is computationally efficient and
thus suitable for  MPC schemes that need to be implemented
distributively or  in hardware with limited computational power.
This algorithm employs parallel block-coordinate updates for the
optimization variables and has similarities to the optimization
algorithm proposed in \cite{Stewart}, but with simpler
implementation, lower iteration complexity  and guaranteed rate of
convergence. We derive (sub)linear rate of convergence for the new
algorithm whose proof relies on the Lipschitz property of the
gradient of the objective function.  The new parallel algorithm is
used for solving MPC problems for general linear  network systems in
a distributed fashion using  local information. For ensuring
stability of the MPC scheme, we use a terminal cost formulation
derived from a distributed synthesis and we eliminate the need for a
terminal state constraint. Compared with the existing approaches
based on an end point constraint, we reduce the conservatism by
combining the underlying structure of the system with distributed
optimization \cite{HuLin:02,LimAla:03,PriNev:00}. Because the MPC
optimization problem is usually terminated before convergence, our
MPC controller is a form of suboptimal control. However, using the
theory of suboptimal control \cite{ScoMay:99} we can still guarantee
feasibility and stability.

This paper is organized as follows.
In Section \ref{secrcd} we derive our
parallel coordinate descent optimization algorithm and prove the
convergence rate for it.
 In Sections
\ref{secpf}-\ref{sec_stability} we introduce the model for general
network  systems, present the MPC problem with a terminal cost
formulation and provide the means for which this terminal cost can
be synthesized distributively.  In Sections \ref{mpc_stab}-\ref{mpc_distrib} we employ
our algorithm for distributively solving MPC problems arising from network systems
and discuss details regarding its implementation.
 In Section \ref{numresults} we compare its
performance with other algorithms and test it on a real application
- a quadruple water tank process.



\section{A parallel  coordinate descent algorithm for smooth convex problems
with separable constraints} \label{secrcd}
We work in  $\rset^n$ composed by column vectors. For $u,v \in
\rset^n$ we denote the standard Euclidean inner product $\langle u,v
\rangle= u^T v$, the Euclidean norm $\left \| u \right
\|=\sqrt{\langle u,u \rangle}$ and $\left \| x \right \|^2_P = x^T P
x $. Further, for a symmetric matrix $P$, we use $P \succ 0 \ (P
\succeq 0)$ for a positive (semi)definite matrix. For  matrices $P$
and $Q$, we use $\text{diag}(P, Q)$ to denote the block diagonal
matrix formed by these two matrices.

In this section we
propose a parallel coordinate descent based algorithm for
efficiently solving the general convex optimization problem of the following
form:
\begin{align}
\label{prob_princ}  f^*=\min_{\mathbf{u}^1 \in \bo{U}^{1},\cdots
,\mathbf{u}^M \in \bo{U}^{M}} f(\mathbf{u}^1,\dots ,\mathbf{u}^M),
\end{align} where $\bo{u}^i \in \rset^{n_\bo{u}^i}$ with $i=1,\dots,M$, are the decision
variables, constrained to individual convex sets $\bo{U}^i \subset \rset^{n_\bo{u}^i}$.
We gather the individual constraint sets $\bo{U}^i$ into the set $\bo{U}=\bo{U}^1\times \dots \times \bo{U}^M$,
and denote the entire decision variable for \eqref{prob_princ} by $\bo{u}=\left [(\mathbf{u}^1)^T \dots (\mathbf{u}^M)^T \right ]^T \in \rset^{n_\bo{u}}$,
with $n_\bo{u}=\sum_{i=1}^M n_\bo{u}^i$.
 As we will show in this section, the new algorithm  can be used  on many parallel
computing architectures, has low computational cost per iteration
and guaranteed convergence rate. We will then apply this algorithm
for solving distributed MPC problems arising in network systems in
Section \ref{pcdmmpc}.


\subsection{Parallel Block-Coordinate Descent Method}
Let us partition the identity matrix in accordance with the structure of the decision
variable $\bo{u}$:
\begin{equation*}
I_{n_\bo{u}}=\left [ (E^1)^T \dots (E^M)^T \right]^T \in \rset^{n_\bo{u} \times
 n_\bo{u}} \text{, }
\end{equation*}
where $E^i \in \rset^{n_\bo{u} \times n_\bo{u}^i}$ for all  $i=1, \cdots, M$.
With matrices $E^i$  we can represent
$\bo{u}=\sum^{M}_{i=1} E^i \bo{u}^i$. We also define the partial
gradient $\nabla_i f(\bo{u}) \in \rset^{n_\bo{u}^i}$ of $f(\bo{u})$ as:
$\nabla_i f(\bo{u})= (E^i)^T \nabla f(\bo{u})$. We assume that the
gradient of $f$ is coordinate-wise Lipschitz continuous with
constants $L_i >0$, i.e:
\begin{equation}
\label{nablai} \left \| \nabla_i f (\bo{u}+{E}^i h_i) - \nabla_i
f(\bo{u}) \right \| \leq L_i \left \| h_i \right \| \;\;\; \forall
\bo{u} \in \rset^{n_\bo{u}}, \; h_i \in \rset^{n_\bo{u}^i}.
\end{equation}
Due to the assumption that $f$ is coordinate-wise Lipschitz continuous,
it can be easily deduced that \cite{Nes:12}:
\begin{equation}\label{ec7}
f(\bo{u}+E^i h_i) \leq f(\bo{u})+ \left <\nabla_i f (\bo{u}), h_i
\right >+\frac{L_i}{2} \left \| h_i \right \|^2 \;\;\; \forall \bo{u} \in
 \rset^{n_\bo{u}}, \;  h_i \in \rset^{n_\bo{u}^i}.
\end{equation}

We now introduce the following norm for the extended space
$\mathbb{R}^{n_\bo{u}}$:
\begin{equation}\label{ec12}
\left\| \bo{u} \right \|_1^2=\sum^{M}_{i=1}L_i \left \| \bo{u}^i \right \|^2,
\end{equation}
which will prove useful for estimating the rate of convergence for
our algorithm.  Additionally, if function $f$ is smooth and
strongly convex with regards to $\left \| \cdot \right \|_1$
 with a parameter $\sigma_1$, then \cite{Nes:04}:
\begin{equation}\label{strong_conv}
f(\bo{w}) \geq f(\bo{v}) + \left < \nabla f(\bo{v}),
\bo{w}-\bo{v}\right > + \frac{\sigma_1}{2} \left \|\bo{w}-\bo{v}
\right \|_1^2 \; \forall \bo{w}, \bo{v} \in \rset^{ n_\bo{u}}.
\end{equation}
Note that if $f$ is strongly convex w.r.t the standard Euclidean
norm $\left \| \cdot \right \|$ with a parameter $\sigma_0$, then
$\sigma_0 \geq \sigma_1  L^i_{\max}$, where
$L^i_{\max}=\displaystyle \max_i L_i$. By taking $\bo{w}=\bo{v}+ E^i
h_i$ and $\bo{v}=\bo{u}$ in \eqref{strong_conv} we also get:
\begin{align*}
f(\bo{u}+E^i h_i) \geq f(\bo{u}) + \left < \nabla_i f(\bo{u}),  h_i
\right > +\frac{\sigma_1 L_i}{2} \left \| h_i \right \|^2 \;\;
\forall \bo{u} \in
 \rset^{ n_\bo{u}}, h_i \in \rset^{n_\bo{u}^i},
\end{align*}
and combining with \eqref{ec7} we also deduce that $\sigma_1 \leq 1$.

We now define the constrained coordinate update for our algorithm:
\begin{align*}
\bo{\bar{v}}^i(\bo{u}) & = \arg \min _{\bo{v}^i \in \bo{U}^{i}}  \left <
\nabla_i f(\bo{u}), \bo{v}^i-\bo{u}^i\right > +\frac{L_i}{2} \left
\| \bo{v}^i - \bo{u}^i \right \|^2 \nonumber \\
\bar{\bo{u}}^i(\bo{u})& = \bo{u}+E^i(\bo{v}^i(\bo{u})-\bo{u}^i), \;
 i=1, \dots, M. \nonumber
\end{align*}

The optimality conditions for the previous optimization problem are:
\begin{align}
 \label{opcond}
& \left< \nabla_i f (\bo{u})+  L_i (\bo{\bar{v}}^i (\bo{u})  -\bo{u}^i),
\bo{v}^i -\bo{\bar{v}}^i(\bo{u})\right >  \geq 0 \;\; \forall \bo{v}^i \in
\bo{U}^{i}.
\end{align}
Taking $\bo{v}^i=\bo{u}^i$ in the previous inequality and combining
with \eqref{ec7} we obtain the following decrease in the objective
function: \setlength{\arraycolsep}{0.0em}
\begin{equation}\label{ec11}
f(\bo{u})-f(\bar{\bo{u}}^i(\bo{u})) \geq \frac{L_i}{2} \left \|
\bo{\bar{v}}^i (\bo{u})-\bo{u}^i \right \|^2.
\end{equation}

We now present our \textit{Parallel Coordinate Descent Method}, that
resembles the method in \cite{Stewart} but with  simpler
implementation, lower iteration complexity and guaranteed rate of
convergence, and is a parallel version of the coordinate descent
method from \cite{Nes:12}:
\begin{center}
\framebox[1.4\width]{
\parbox{8cm}{
\begin{center}
\textbf{ Algorithm PCDM }
\end{center}
\textbf{Choose} $\bo{u}_0^i \in \bo{U}^i$ \textbf{for all}
$i=1,\dots, M$. \textbf{For} $k\geq 0$:
\begin{enumerate}
\item{\textbf{Compute in parallel $\bo{\bar{v}}^i(\bo{u}_k), \; i=1,\dots,M.$
}}
\item{\textbf{Update in parallel:}
\begin{equation*}
 \bo{u}_{k+1}^i = \frac{1}{M} \bo{\bar{v}}^i(\bo{u}_k)+\frac{M-1}{M}\bo{u}_k^i, \; i=1,\dots,M.
\end{equation*}
}
\end{enumerate}
}}
\end{center}

Note that if the sets $ \bo{U}^{i}$ are simple (by simple we
understand that the projection on these sets is easy), then
computing $\bo{\bar{v}}^i(\bo{u})$ consists of projecting a vector
on these sets and can be done numerically very efficient.  For
example, if these sets are simple box  sets, i.e $\bo{U}^{i}=\left
\{ \bo{u}^i \in \rset^{n_\bo{u}^i} | \bo{u}^i_{min} \leq \bo{u}^i
\leq \bo{u}^i_{max} \right \}$, then the complexity of computing
$\bo{\bar{v}}^i(\bo{u})$,  once $\nabla_i f(\bo{u})$ is available,
is $\mathcal{O}(n_\bo{u}^i)$. In turn, computing $\nabla_i
f(\bo{u})$ has, in the worst case, complexity $\mathcal{O}
(n_\bo{u}^i n_\bo{u})$ for quadratic dense functions. Thus,
Algorithm PCDM has usually a very low iteration cost per subsystem
compared to other existing methods, e.g. Jacobi type algorithm
presented in \cite{Stewart}, which usually require numerical
complexity at least $\mathcal{O}((n_\bo{u}^i)^3 + n_\bo{u}^i
n_\bo{u})$ per iteration for each subsystem $i$, provided that the
local quadratic problems are solved with an interior point solver.
Also, in the following two theorems we provide estimates for the
convergence rate of our algorithm, while for the algorithm in
\cite{Stewart} only asymptotic convergence is proved.

From \eqref{opcond}-\eqref{ec11}, convexity of $f$ and $\bo{u}_{k+1}
= \sum_i \frac{1}{M} \bar{\bo{u}}^i(\bo{u}_k)$ we see immediately
that method PCDM decreases strictly the objective function at
each iteration, provided that $\bo{u}_{k}\neq \bo{u}_*$, where $\bo{u}_*$ is the optimal solution
of \eqref{prob_princ}, i.e.:
\begin{align}
\label{decreasef} f(\bo{u}_{k+1}) < f(\bo{u}_{k}) \;\; \forall k
\geq 0, \; \bo{u}_{k} \neq \bo{u}_*.
\end{align}

Let $f^*$ be the optimal value in optimization problem
\eqref{prob_princ}. The following theorem derives  convergence rate
of Algorithm PCDM and employs standard techniques for proving  rate
of convergence of the gradient method \cite{Nes:12,Nes:04}:
\begin{theorem}
\label{thc} If function $f$ in optimization problem
\eqref{prob_princ} has a coordinate-wise Lipschitz continuous
gradient with constants $L_i$ as given in \eqref{nablai}, then
Algorithm PCDM has the following sublinear rate of convergence:
\begin{equation*}
f(\bo{u}_{k})\! - \!f^*\leq \frac{M}{M+k} \left(
\frac{1}{2}r_0^2+f(\bo{u}_0)\! - \!f^* \right ),
\end{equation*}
\end{theorem}
where $r_0= \left \| \bo{u}_0 \!-\!\bo{u}_* \right \| _1$.
\begin{proof}
We introduce the following term:
\begin{equation*}
r_k^2= \left \| \bo{u}_k \!-\!\bo{u}_* \right \| _1^2=
\sum_{i=1}^{M} L_i \left < \bo{u}_k^i \! - \!\bo{u}_*^i,
\bo{u}_k^i\! - \!\bo{u}_*^i \right >,
\end{equation*}
where $\bo{u}_*$ is the optimal solution of \eqref{prob_princ} and
$\bo{u}^i_*=(E^i)^T\bo{u}_*$. Then, using similar derivations as in \cite{Nes:12}, we have:
\begin{align*}
r_{k+1}^2\!& \! =\! \sum_{i=1}^{M} L_i \left \| \frac{1}{M}
\bo{\bar{v}}^i(\bo{u}_k)+(1\! -
 \!\frac{1}{M})\bo{u}_k^i\! - \! \bo{u}_*^i \right \|^2 \nonumber \\
&\overset{(\ref{opcond})}{\leq}   r_k^2 + \sum_{i=1}^{M}
\frac{L_i}{M}(\frac{1}{M}\! - \!2) \left \|
\bo{\bar{v}}^i(\bo{u}_k)-\bo{u}^i_k\right \| ^2 + \frac{2}{M} \left < \nabla_i f(\bo{u}_k), \bo{u}_*^i\! - \!\bo{\bar{v}}^i(\bo{u}_k) \right > \nonumber \\
&\overset{\frac{1}{M} \leq 1}\leq r_k^2 -  \frac{2}{M}
\sum_{i=1}^{M} (\frac{L_i}{2} \left \|
\bo{\bar{v}}^i(\bo{u}_k)-\bo{u}^i_k\right \| ^2 +\nonumber \\
& \quad \quad \left < \nabla_i f(\bo{u}_k),
\bo{\bar{v}}^i(\bo{u}_k)-\bo{u}^i_k \right >   +  \left < \nabla_i
f(\bo{u}_k), \bo{u}_*^i\! - \!\bo{u}_k^i \right >). \nonumber
\end{align*}
By convexity of $f$ and \eqref{ec7} we obtain:
\begin{align}
\label{rez_cent} r_{k+1}^2 \leq r_k^2 -  &
2(f(\bo{u}_{k+1})-f(\bo{u}_k))+ \frac{2}{M} \left <  \nabla
f(\bo{u}_k), \bo{u}_*-\bo{u}_k \right
>
\end{align}
and adding up these inequalities we get:
\begin{align*}
\frac{1}{2} r_0^2\!+\!f(\bo{u}_0) \!-\! f^* & \!\geq\!  \frac{1}{2}
r_{k+1}^2 \!\!+\! f(\bo{u}_{k+1}) \!-\!f^* \!+\!
 \!\frac{1}{M} \!\! \sum_{j=0}^{k} \!(f(\bo{u}_j) \!-\! f^*) \\
& \geq f(\bo{u}_{k+1})\! - \!f^*+\frac{1}{M} \sum_{j=0}^{k}
(f(\bo{u}_j)\! - \!f^*).
\end{align*}
Taking into account that our algorithm is a descent algorithm, i.e.
$f(\bo{u}_j)\geq f(\bo{u}_{k+1})$ for all $j \leq k$ and by the
previous inequality the proof is complete. \qed
\end{proof}

Now, we derive linear convergence rate for Algorithm PCDM, provided
that  $f$ is additionally  strongly convex:
\begin{theorem}
\label{thc2} Under the assumptions of Theorem \ref{thc} and if we
further assume that $f$ is strongly convex with regards to $\left \|
\cdot \right \|_1$ with a  constant $\sigma_1$ as given in
\eqref{strong_conv}, then the following linear rate of convergence
is achieved for Algorithm PCDM:
\begin{align*}
f(\bo{u}_{k}) \!-\! f^*  \!\leq\! \left ( 1 \!-\!
\frac{2\sigma_1}{M(1+\sigma_1)} \right )^{k} \!\! \left( \frac{1}{2}
r_0^2 \!+\! f(\bo{u}_{0}) \!- \! f^*  \right).
\end{align*}
\end{theorem}
\begin{proof}
We take $\bo{w}= \bo{u}^*$ and $\bo{v}=\bo{u}_k$
in \eqref{strong_conv} and through \eqref{rez_cent} we get:
\begin{align}\label{rezult1}
\frac{1}{2} r_{k+1}^2 + f(\bo{u}_{k+1})-f^* \leq &\frac{1}{2} r_k^2
+f(\bo{u}_k)-f^* -\frac{1}{M} (f(\bo{u}_k)-f^* +\frac{\sigma_1}{2}
r_k^2).
\end{align}
From the strong convexity of $f$ in \eqref{strong_conv} we also get:
\begin{align*}
f(\bo{u}_k)-f^* + \frac{\sigma_1}{2} r_k^2 \geq \sigma_1 r_k^2.
\end{align*}
We now define $\gamma=\frac{2\sigma_1}{1+\sigma_1} \in [0,1]$ and using
the previous inequality  we obtain the following result:
\begin{align*}
 f(\bo{u}_k)-f^*+\frac{\sigma_1}{2} r_k^2 \geq
 &\gamma \left( f(\bo{u}_k)-f^*+\frac{\sigma_1}{2} r_k^2 \right) + (1-\gamma)\sigma_1 r_k^2.
\end{align*}
Using this inequality in \eqref{rezult1} we get:
\begin{align*}
\frac{1}{2} r_{k+1}^2 \!+\! f(\bo{u}_{k+1})-f^*  \!\leq\! \left ( 1-
\frac{\gamma}{M} \right ) \left ( \frac{1}{2} r_k^2 +f(\bo{u}_{k})
-f^*  \right ).
\end{align*}
Applying this inequality iteratively, we obtain the following for $k \geq 0$:
\begin{align*}
\frac{1}{2} r_{k}^2 \!+\! f(\bo{u}_{k}) \!-\! f^* \!\leq\! \left (1
\!-\! \frac{\gamma}{M} \right )^{k} \left ( \frac{1}{2} r_0^2
+f(\bo{u}_{0}) - f^*  \right ),
\end{align*}
and by replacing $\gamma=\frac{2\sigma_1}{1+\sigma_1} $ we complete the
proof. \qed
\end{proof}

The following properties follow immediately for our Algorithm PCDM.
\begin{lema}
\label{lemmap} For the optimization problem \eqref{prob_princ}, with
the assumptions of Theorem \ref{thc2}, we have the following statements:\\
 (i) Given any feasible initial guess $\bo{u}_0$, the iterates of the
Algorithm PCDM are feasible at each iteration, i.e. $\bo{u}_k^i \in
\bo{U}^{i}$ for all $k \geq 0$.\\
(ii) The  function $f$ is nonincreasing, i.e. $f( \bo{u}_{k+1} )
\leq f(\bo{u}_{k})$ according to \eqref{decreasef}. \\
(iii) The sub(linear) rate of convergence of Algorithm PCDM is
given in Theorem \ref{thc} (Theorem \ref{thc2}).
\end{lema}


\section{Application of Algorithm PCDM to distributed suboptimal
MPC} \label{pcdmmpc} The Algorithm PCDM can be used to solve
distributively input constrained MPC problems for network systems after state elimination.
In this section we show that the MPC scheme obtained
by solving approximately the corresponding optimization problem with
Algorithm PCDM is stable and distributed.

\subsection{MPC for network systems with  terminal cost and  without end constraints}
\label{secpf}

In this paper we consider discrete-time network systems, which  are
usually modeled  by a graph  whose nodes represent subsystems and
whose arcs indicate dynamic couplings,  defined by the following
linear state  equations \cite{CamSch:11,DoaKev:11,Sca:09}:
\begin{align}
x^i_{t+1}= \sum _{j \in \neigh } \label{mod3}
 A^{ij} x_t^{j} +B^{ij} u_t^j, \qquad i=1, \cdots, M,
\end{align}
where $M$ denotes the number of interconnected subsystems,  $x_t^j
\in \mathbb{R}^{n_j}$ and $u_t ^j \in \mathbb{R}^{m_j}$ represent
the state and respectively the input of $j$th subsystem at time $t$,
$A^{ij} \in \mathbb{R}^{n_i \times n_j}$, $B^{ij} \in
\mathbb{R}^{n_i \times m_j}$ and
  $\neigh$ is the set of indices
which contains the index $i$ and that of its neighboring subsystems.
A particular case of  \eqref{mod3}, that is frequently found in
literature \cite{Nec:08,Stewart,Venkat}, has the following dynamics:
\begin{equation}
\label{simp1} x^i_{t+1}=A^{ii} x_t ^i+   \sum _{j \in \neigh
} B^{ij} u_t^j.
\end{equation}
For stability analysis, we also express the dynamics of the entire
system: $x_{t+1}=A x_t+ Bu_t$, where $n=\displaystyle \sum_{i=1}^M
n_i$, $m=\displaystyle \sum_{i=1}^M m_i$, $x_t \in \rset^n$,  $u_t
\in \rset^{m}$ and $A \in \rset^{n \times n}$, $B\in \rset^{n \times
m}$. For system  \eqref{mod3} or \eqref{simp1} we consider local
input constraints:
\begin{equation}\label{mod2}
u_t^i \in U^i \;\;\;  \; i=1,\cdots,M, \;\ t\geq 0,
\end{equation}
with $U^i \subseteq \mathbb{R}^{m_i}$ compact, convex sets with the
origin in their interior.  We also consider convex local stage and
terminal costs for each subsystem $i$: $\ell^i(x^i, u^i)$ and
$\ell_{\text{f}}^i(x^i) $. Let us denote the input trajectory for
subsystem $i$ and the overall input trajectory for the entire system
by:
\begin{align*}
\bo{u}^i&=[(u^i_0)^T \cdots (u^i_{N-1})^T]^T, \;
\bo{u}=[(\bo{u}^1)^T \cdots (\bo{u}^M)^T]^T.
\end{align*}
 We can now formulate the MPC problem for
system \eqref{mod3} over a prediction horizon of length $N$ and a
given initial state $x$ as \cite{RawMay:09}:
\begin{align}
V_N^*(x) =&  \min _{u_t^i \in U^i\; \forall i, t} V_N(x, \bo{u}) \quad
\left(:= \sum_{i=1}^M \sum_{t=0}^{N-1}
\ell^i(x_t^i,u_t^i) + \ell_{\text{f}}^i (x_N^i) \right) \label{mod5} \\
&\text{s.t:} \; x^i_{t+1}= \!\!\! \sum _{j \in
 \neigh } A^{ij} x_t^{j} + B^{ij} u_t^j, \;\;  x_0^i=x^i , \;  i= 1, \cdots, M,
\;\; t \geq 0. \nonumber
\end{align}
It is well-known that by eliminating the states using dynamics
\eqref{mod3}, the MPC problem \eqref{mod5} can be recast
\cite{RawMay:09} as a convex optimization problem of type
\eqref{prob_princ}, where $n_\bo{u}^i=Nm_i$, the function $f$ is
convex (recall that we assume the stage
 and final costs $\ell^i(\cdot)$ and $\ell^i_\text{f}(\cdot)$ to be convex), whilst the convex
sets $\bo{U}^{i}$ are the  Cartesian product of the convex sets
$U^i$ for $N$ times. Moreover, in the case  of dynamics
\eqref{simp1}, we can express the objective function of problem
\eqref{prob_princ} as a sum of local functions with sparse
structure:
\begin{equation}\label{simp55}
f(\mathbf{u}^1,\dots ,\mathbf{u}^M)=\sum_{i=1}^{M}
f^i(\mathbf{u}^j,j \in \neigh).
\end{equation}
Further, we denote  the approximate solution produced by Algorithm
PCDM for problem \eqref{mod5} after certain number of iterations
with $\bo{u}^{\text{CD}}$. We also consider that at each MPC step
the Algorithm PCDM is initialized (warm start) with the shifted
sequence of controllers obtained at the previous step and the
feedback controller $\kappa(\cdot)$ computed in Section
\ref{sec_stability} below. The suboptimal MPC scheme corresponding
to \eqref{mod5} would now be:
\begin{center}
\framebox[1.3\width]{
\parbox{9cm}{
\begin{center}
\textbf{ Suboptimal MPC scheme}
\end{center}
\textbf{Given initial state} $x$   \textbf{and initial} $ \tilde{\bo{u}}^{\text{CD}}$
 \textbf{repeat}:
\begin{enumerate}
\item{ \textbf{Recast MPC problem \eqref{mod5} as opt. problem \eqref{prob_princ} }}
\item{ \textbf{Solve \eqref{prob_princ} approximately with Alg. PCDM starting
from  $ \tilde{\bo{u}}^{\text{CD}}$   and
obtain~$\bo{u}^{\text{CD}}$ } }
\item{\textbf{Update $x$. Update $\tilde{\bo{u}}^{\text{CD}}$  using warm start.} }
\end{enumerate}
}}
\end{center}


\subsection{Distributed synthesis for a terminal cost}
\label{sec_stability} We assume that stability of the MPC scheme
\eqref{mod5} is enforced by adapting the terminal cost
$\ell_{\text{f}} (\cdot) = \sum_{i=1}^M \ell_{\text{f}}^i (\cdot)$
and the horizon length $N$ appropriately such that sufficient
stability criteria are fulfilled
\cite{HuLin:02,LimAla:03,PriNev:00}. Usually, stability of MPC with
quadratic stage cost  $\ell^i(x^i, u^i) =  \left \| x^i \right
\|^2_{Q^i}  +  \left \| u^i \right \|^2_{R^i} $, where the matrices
$Q^i \succeq 0$ and $R^i \succ 0$, and without terminal constraint
is enforced if the following criteria hold: there exists a
neighborhood of the origin $\Omega \subseteq \rset^n$, a stabilizing
feedback law $\kappa(\cdot)$ and a terminal cost $\ell_{\text{f}}
(\cdot)$ such that we have
\begin{equation}
\begin{cases}
 \ell_\text{f} (Ax+ B \kappa(x)) - \ell_\text{f}(x) +  \kappa(x)^T R
 \kappa(x) +
 x^TQx & \leq 0 \;\; \forall x \in \Omega \\
 \kappa(x) \in \bo{U}, \; Ax+B \kappa(x) \in \Omega,
\end{cases} \label{lyap1}
\end{equation}
where the matrices $Q$ and $R$ have a block diagonal structure and
are composed of the blocks $Q^i$ and $R^i$, respectively. As shown
in \cite{HuLin:02,LimAla:03,PriNev:00}, MPC schemes based on  the
condition \eqref{lyap1} are  usually less conservative than schemes
based on end point constraint. Keeping in line with the distributed
nature of our system, the control law $\kappa(\cdot)$ and the final
stage cost $\ell_\text{f}(\cdot)$ need to be computed locally. In
this section we develop a distributed synthesis procedure to
construct them locally. We choose the terminal cost for each
subsystem $i$ to be quadratic: $\ell_\text{f}^i (x_N^i)=\left \|
x_N^i\right \|_{P^i}^2$, where $P^{i} \succ 0$.  For a locally
computed $\kappa(\cdot)$, we employ distributed control laws:  $u^i
= F^i x^i$, i.e. $\kappa(\cdot)$ is taken linear with a
block-diagonal structure. Centralized LMI formulations of
\eqref{lyap1} for quadratic terminal costs are well-known in the
literature \cite{RawMay:09}. However, our goal is to  solve
\eqref{lyap1} distributively. To this purpose, we first need to
introduce vectors $x^{\neigh} \in \rset^{n_{\neigh}}$ and
$u^{\neigh} \in \rset^{m_{\neigh}}$ for subsystem $i$, where
$n_{\neigh}= \displaystyle \sum_{j \in \neigh} n_j$ and $m_{\neigh}=
\displaystyle \sum_{j \in \neigh} m_j$. These vectors are comprised
of the state and input vectors of
 subsystem $i$ and those of its neighbors: $x^{\neigh} = \begin{bmatrix} (x^j)^T, j \in \neigh
 \end{bmatrix}^T, \; u^{\neigh} = \begin{bmatrix} (u^j)^T, j \in \neigh  \end{bmatrix}^T$.

Since our synthesis procedure needs to be distributed and taking
into account that $\ell_{\text{f}} (\cdot) = \sum_{i=1}^M
\ell_{\text{f}}^i (\cdot)$, we impose the following distributed
structure to ensure \eqref{lyap1} (see also \cite{JokLaz:09} for a
similar approach where infinity-norm control Lyapunov functions are
synthesized in a decentralized fashion   by solving  linear programs
for each subsystem)  for $i=1, \cdots, M$:
\begin{align}
\ell_\text{f}^i ((x^i)^+) - & \ell_\text{f}^i ((x^i)) \!+\!  ( F^i
x^i)^T R^i F^i x^i\! +\! (x^i)^T Q^i x^i  \leq q^i(x^{\neigh}) \;
\forall x^{\neigh} \!\! \in \rset^{n_{\neigh}} \label{ineq_individ}
\end{align}
such that $q(x) = \sum_{i=1}^M   q^i(x^{\neigh}) \leq 0$. We assume
that  $q^i(x^{\neigh})$ also have a quadratic form, with
$q^i(x^{\neigh}) = \left \| x^{\neigh} \right \|_{W^{\neigh}}^2$,
where $W^{\neigh} \in \rset^{n_{\neigh} \times n_{\neigh}}$. Being a
sum of quadratic functions, $q(x)$ can itself be expressed as a
quadratic function, $q(x)=\left \| x \right \|^2_W$, where $W \in
\rset^{n \times n}$  is formed from the appropriate block components
of matrices $W^{\neigh}$. Note that we do not require that matrices
$W^{\neigh}$ be negative semidefinite. On the contrary, positive or
indefinite matrices allow local terminal costs to increase so long
as the global cost still decreases.  This approach reduces the
conservatism in deriving the matrices $P^i$ and $F^i$. For obtaining
$P^i$ and $F^i$, we introduce matrices $E^i_n \in \rset^{n_i \times
n}$, $E^i_m \in \rset^{m_i  \times m}$, $\Jvec_n \in
\rset^{n_{\neigh}  \times n}$, $\Jvec_m \in \rset^{m_{\neigh} \times
m}$ such that $x^i=E^i_n x, \; u^i=E^i_m u, \; x^{\neigh}=\Jvec_n x$
and $ u^{\neigh}=\Jvec_m u$.  We now define the matrices
$A^{\neigh}=E^i_n A (\Jvec_n)^T$, $\ B^{\neigh}= E^i_n B
(\Jvec_m)^T$ and $ F^{\neigh}=\Jvec_m F(\Jvec_n) ^T$, as to express
the dynamics \eqref{mod3} for subsystem $i$: $x_{t+1}^i=
(A^{\neigh}+B^{\neigh} F^{\neigh})x^{\neigh}_t$. Using these
notations we can now recast inequality \eqref{ineq_individ}~as:
\begin{align}
&(A^{\neigh}+B^{\neigh} F^{\neigh})^T P^i (A^{\neigh}+B^{\neigh} F^{\neigh}) \label{ineq_system} \\
&-\Jvec_n (E^i_n)^T (P^i+Q^i+(F^i)^T R^i F^i) E^i_n (\Jvec_n)^T
\preceq W^{\neigh}. \nonumber
\end{align}

The task of finding suitable $P^i$, $F^i$ and $W^{\neigh}$ matrices
is now reduced to the following optimization problem:
\begin{align}
& \underset{P^i,F^i,W^{\neigh},\delta} \min  \{ \delta: \;\;
\text{MI} \ \eqref{ineq_system}, \;  i=1,\cdots,M, \quad W
\preceq \delta I \}. \label{sdp_1}
\end{align}
It can be easily observed that if the optimal value $\delta^* \leq
0$, consequently $W \leq 0$ and \eqref{lyap1} holds. This
optimization problem, in its current  nonconvex form, cannot be
solved efficiently. However, it can be recast as a sparse SDP if we
can reformulate \eqref{ineq_system} as an LMI. We need now to make
the assumption that all the subsystems have the same dimension for
the states, i.e. $n_i = n_j$ for all $i, j$. Subsequently, we
introduce the well-known  linearizations: $P^i=(S^i)^{-1}, \;
F^i=Y^i G^{-1}$ and a series of matrices that will be of aid in
formulating the LMIs:
\begin{align*}
G^{\neigh}&= I_{\left | {\neigh} \right |} \otimes G, \;\;  G^{\neigh
 \setminus i}=[0 \;\; I_{\left | {\neigh} \right | -1} \otimes G ], \; S^{\neigh} = \text{diag}
 (S^i, \mu_i  I_{({n_{\neigh} -n_i})})  \\
Y^{i,j}&\!=\!F^j \! G, \; j \in \neigh \!\! \setminus \!i, \;\;
Y^{\neigh}\!\!\!=\!
\text{diag}(Y^i\!,\! Y^{i,j})\!=\!F^{\neigh}\!\! G^{\neigh}, \\
T^{\neigh}&\! =\!\begin{bmatrix}
A^{\neigh}G^{\neigh} \!+\! B^{\neigh} Y^{\neigh} \\
G^{\neigh \setminus i}
  \end{bmatrix}\!,\; T^i\!=\!\begin{bmatrix}
 (Q^i)^{\frac{1}{2}} G & \; \; 0  \\
(R^i)^{\frac{1}{2}}Y^i & \; \; 0
 \end{bmatrix},
\end{align*}
where the $0$ blocks are of appropriate dimensions\footnote{By $I_n$
we denote the identity matrix of size $n \times n$,  by $\otimes$ we
denote the standard Kronecker product and by $\left | {\neigh}
\right |$ the cardinality of the set $\neigh$.}.

\begin{lema}
\label{lemma1} If the following SDP:
\begin{equation}\label{sdp_2}
\underset{G,S^i,Y^i,Y^{i,j},\tilde{W},\mu^i,\delta}{\min} \; \delta  \\
\end{equation}
\begin{align}
&\text{ s.t:} \quad
\begin{bmatrix}
G^{\neigh}\!+\!(G^{\neigh})^T \!-\!S^{\neigh}\!+ \! \tilde{W}^{\neigh}& \! *  \! &  \! *\! \\
\!T^{\neigh}\! \!&\! S^{\neigh} & * \\
\!T^i\! & \!0\! & \!I\!
\end{bmatrix} \succ 0 \label{lmisystem}\\
&Y^{i,j}=Y^{j}   \;\; \forall j\in \neigh, \;\  i=1,\cdots,M, \quad
\tilde{W} \preceq \delta I, \nonumber
\end{align}
has an optimal value $\delta^* \leq 0$,  then \eqref{lyap1}
holds\footnote{By $*$ we denote the transpose of the symmetric block
of the matrix.}.
\end{lema}

\begin{proof}
From \eqref{lmisystem} we observe that $S^{\neigh} \succ 0$, so that
$(S^{\neigh}-G^{\neigh})^T (S^{\neigh})^{-1} (S^{\neigh} \\ -
G^{\neigh}) \succeq 0$, which in turn implies
\begin{align}\label{g_rel}
G^{\neigh}+(G^{\neigh})^T-S^{\neigh}\preceq
(G^{\neigh})^T(S^{\neigh})^{-1}G^{\neigh}.
\end{align}
If we apply the Schur complement to \eqref{lmisystem}, we obtain:
\begin{align*}
 0\preceq & G^{\neigh}+(G^{\neigh})^T-S^{\neigh}+\tilde{W}^{\neigh} - (T^{\neigh})^T (S^{\neigh})^{-1} T^{\neigh} - (T^i)^T T^i
\end{align*}
and by \eqref{g_rel} we get $(G^{\neigh})^{-T} \left [
(T^{\neigh})^T (S^{\neigh})^{-1} T^{\neigh} + (T^i)^T T^i \right ]
(G^{\neigh})^{-1} -(S^{\neigh})^{-1} \\  \preceq (G^{\neigh})^{-T}
\tilde{W}^{\neigh} (G^{\neigh})^{-1}$, which is equivalent to
\eqref{ineq_system} if  we consider  $W^{\neigh} = (G^{\neigh})^{-T}
\tilde{W}^{\neigh}  (G^{\neigh})^{-1}$. \qed
\end{proof}
There exist in  literature  many optimization algorithms (see e.g.
\cite{Nay:07}) for solving  distributively sparse SDP problems in
the form \eqref{sdp_2}.

\subsection{Stability of the MPC scheme}
\label{mpc_stab}

 We can
consider the cost function of the MPC problem
$V_N(x,\bo{u}^{\text{CD}})$ as a Lyapunov function, using the
standard theory for suboptimal control (see e.g.
\cite{LimAla:03,PriNev:00,RawMay:09,ScoMay:99,Stewart} for similar
approaches). We also consider that at each MPC step the  Algorithm
PCDM is initialized (warm start) with the shifted sequence of
controllers obtained at the previous step and the feedback
controller $\kappa(\cdot)$ computed in Section \ref{sec_stability}
such that \eqref{lyap1} is satisfied  and we denote it by
$(\tilde{\bo{u}}^{\text{CD}})^+$. Assume also that $\kappa(\cdot)$,
$\ell_f(\cdot)$ and $\alpha>0$ are chosen such that, together with
the following set
\begin{equation*}
\Omega=\left \{ x \in \rset^n: \ell_f(x)\leq \alpha \right \},
\end{equation*}
 satisfies \eqref{lyap1}. Then, using Theorem 3 from
\cite{LimAla:03} we have that our MPC controller stabilizes
asymptotically the system for all initial states $x \in X_N$, where
\begin{equation*}
 X_N=\left \{ x \in \rset^n: V_N^*(x)\leq Nd+\alpha \right \},
\end{equation*}
such that $V_N(x,\bo{u}^{CD})\leq Nd+\alpha$, where $d>0$ is  a
parameter for which we have $\ell(x,\bo{u})\geq d$  for all $ x
\notin \Omega$. Clearly, this MPC scheme is locally stable with a
region of attraction $X_N$.

\subsection{Distributed implementation of the MPC scheme based on Algorithm PCDM}
\label{mpc_distrib}
In this section we discuss some technical aspects for the
distributed  implementation of the  MPC scheme derived above when
using Algorithm PCDM to solve the control problem \eqref{mod5}.
Usually, in the linear MPC framework, the local stage and  final
cost are taken of the following quadratic form:
\begin{align*}
\ell^i(x^i, u^i) = \left \| x^i \right \|^2_{Q^i}  + \left \| u^i
\right \|^2_{R^i}, \quad \ell^i_{\text{f}} (x^i) = \left \| x^i
\right \|^2_{P^i},
\end{align*}
where the matrices $Q^i, P^i \in \rset^{n_{i} \times n_{i}}$ are
positive semidefinite, whilst matrices  $R_i \in \rset^{m_{i} \times
m_{i}} $  are positive definite. We also assume that the local
constraints sets $U^i$ are polyhedral. In this particular case, the
objective function in \eqref{mod5}, after eliminating the dynamics,
is quadratically strongly convex, having the form \cite{RawMay:09}:
 \[ f(\bo{u}) =0.5 \ \bo{u}^T \bo{Q} \bo{u} + (\bo{W} x + \bo{w})^T  \bo{u}, \] where
$\bo{Q}$ is positive definite  due to the assumption that all $R^i$
are positive definite.  Usually, for the dynamics \eqref{mod3} the
corresponding matrices $\bo{Q}$ and $\bo{W}$ obtained after
eliminating the states  are dense and despite the fact that
Algorithm PCDM can perform parallel computations (i.e. each
subsystem needs to solve small local problems) we need all to all
communication between subsystems. However, for the dynamics
\eqref{simp1} the corresponding matrices $\bo{Q}$ and $\bo{W}$ are
sparse and in this case in our Algorithm PCDM we can perform
distributed computations (i.e. the subsystems solve small local
problems in parallel and they need to communicate only with their
neighborhood subsystems as detailed below). Indeed, if the dynamics
of the system are given by \eqref{simp1}, then \[ x^i_{t+1} =
(A^{ii})^t x^i_t + \sum_{l=1}^t \sum _{j \in \mathcal {N}^{i}}
(A^{ii})^{l-1} B^{ij} u^j_{t-l} \] and thus the matrices $\bo{Q}$
and $\bo{W}$ have a sparse structure (see also \cite{CamSch:11}).
Let us define the \textit{neighborhood subsystems} of a certain
subsystem $i$ as $ \hat {\mathcal N}^i = {\mathcal N}^i \cup \{l: \;
l \in {\mathcal N}^j, j \in \bar {\mathcal N}^i\}$, where $\bar
{\mathcal N}^i = \{ j: \; i \in {\mathcal N}^j \}$, then the matrix
$\bo{Q}$ has all the $(i,j)$ block matrices $\bo{Q}^{ij} = 0$ for
all $j \notin  \hat {\mathcal N}^i$ and the matrix $\bo{W}$ has all
the block matrices $\bo{W}^{ij} = 0$ for all $j \notin  \bar
{\mathcal N}^i$, for any given subsystem $i$. As a result, we can
express the objective function of problem \eqref{prob_princ} as a
sum of local functions with sparse structure:
\begin{equation}\label{simp5}
f(\mathbf{u}^1,\dots ,\mathbf{u}^M)=\sum_{i=1}^{M}
f^i(\mathbf{u}^j,j \in \neigh).
\end{equation}
Thus, the $i$th block components of $\nabla f$
 can be computed using only local information:
\begin{align}
 \nabla_i f(\bo{u}) &= \sum_{j \in \hat {\mathcal N}^{i}} \bo{Q}^{ij}
\bo{u}^j +  \sum_{j \in \bar {\mathcal N}^{i}}  \bo{W}^{ij} x^j +
\bo{w}^i. \label{grad_L}
\end{align}

Note that in  Algorithm PCDM the only parameters that we need to
compute are the  Lipschitz constants $L_i$. However, in the MPC
problem, $L_i$ does not depend on the initial state $x$ and can be
computed locally by each subsystem as: $L_i =
\lambda_{\max}(\bo{Q}^{ii})$. From the previous discussion  it
follows immediately that the iterations of Algorithm PCDM can be
performed in parallel using  distributed computations (see
\eqref{grad_L}).

Further, our Algorithm PCDM has a simpler implementation of the
iterates than the algorithm from \cite{Stewart}: in Algorithm PCDM
the main step consists of computing  local  projections on the sets
$\bo{U}^i$ (in the context of  MPC  usually these sets are simple
and the projections can be computed in closed form); while in the
algorithm from \cite{Stewart} this step is replaced with solving
local dense QP problems with the feasible set given by $\bo{U}^i$
(even in the context of MPC  this local QP problems cannot be solved
in closed form and an additional QP solver needs to be used).
Finally,  the number of iterations for finding an approximate
solution can be easily predicted in our Algorithm (see Theorems
\ref{thc} and \ref{thc2}), while in the algorithm from
\cite{Stewart} the authors prove only asymptotic converge.


\section{Numerical Results}
\label{numresults} \noindent Since our Algorithm PCDM has
similarities with the algorithm from \cite{Stewart}, in this section
we compare these two algorithms  on controlling a laboratory setup with DMPC
(4 tank process) and on  MPC problems for random network systems of varying
dimension.

\subsection{Quadruple tank process}\label{quadtank}

\begin{figure}[h]
\begin{center}
 \includegraphics[scale=0.8]{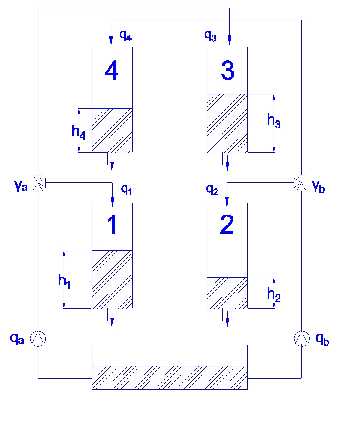}
\caption{Quadruple tank process diagram.} \label{procdiag}
\end{center}
\end{figure}
To demonstrate the applicability of our Algorithm PCDM, we apply
this newly developed method for solving the optimization problems
arising from the MPC problem for a process consisting of four
interconnected water tanks, see Fig. \ref{procdiag} for the process
diagram, whose objective is to control the level of water in each of
the four tanks. For this plant, there are two types of system inputs
that can be considered: the pump flows, when the ratios of the three
way valves are considered fixed, or the ratios of the three way
valves, whilst having fixed flows from the pumps. In this paper, we
consider the latter option, with the valve ratios denoted by
$\gamma_a$ and $\gamma_b$, such that tanks 1 and 3 have inflows
$\gamma_a q_a$ and $(1-\gamma_a) q_a$, while tanks 2 and 4 have
inflows $\gamma_b q_b$ and $(1-\gamma_b)q_b$. The simplified
continuous nonlinear model of the plant is well known
\cite{AlvLim:11}. We use the following notation:  $h_i$ are the
levels and $a_i$ are the  discharge constants of tank $i$, $S$ is
the cross section of the tanks, $\gamma_a$, $\gamma_b$  are the
three-way valve ratios, both in $[0,1]$, while $q_a$ and $q_b$ are
the pump flows.


\begin{table}[h]
 \centering
\footnotesize
\setlength{\tabcolsep}{3pt}
\begin{tabular}{|c|c|c|c|c|c|c|c|c|c|c|c|c|} \hline
 Param & S & $a_1$ & $a_2$ & $a_3$ & $a_4$ & $h_1^0$ &  $h_2^0$ &  $h_3^0$ &  $h_4^0$ &  $q_{a/b}^{max}$ &  $\gamma_a^0$ & $\gamma_b^0$ \\
 \hline Value & 0.02 & $5.8e{-5}$ &  $6.2e{-5}$ & $2e{-5}$ & $3.6e{-5}$ &  $0.19$ & $0.13$ & $0.23$ & $0.09$ &  $0.39$ & 0.58 & 0.54   \\
\hline Unit & $m^2$ & $m^2$ & $m^2$ & $m^2$ & $m^2$ & $m$ & $m$ & $m$ & $m$  & $\frac{m^3}{h}$  & &  \\ \hline
\end{tabular}
\caption{Quadruple tank process parameters.}
\label{tab:4tank parameters}
\end{table}
\normalsize

The discharge constants $a_i$, with $i=1,\dots,4$ and the other
parameters of the model are determined experimentally from our
laboratory setup (see Table \ref{tab:4tank parameters}). We can
obtain a linear continuous state-space model by linearizing the
nonlinear model at an operating point given by $h_i^0$,
$\gamma_a^0$, $\gamma_b^0$, and the maximum inflows from the pumps,
with the deviation variables $x^i=h_i-h_i^0$,
$u^1=\gamma_a-\gamma_a^0$, $u^2=\gamma_b-\gamma_b^0$:
\begin{align*}
 \frac{dx}{dt}=
\begin{bmatrix} -\frac{1}{\tau_1} & 0 &  0 & \frac{1}{\tau_4} \\
 0 & -\frac{1}{\tau_2} &  \frac{1}{\tau_3} & 0  \\ 0 & 0 & -\frac{1}{\tau_3} & 0 \\
0 & 0 & 0 & -\frac{1}{\tau_4} \end{bmatrix} x + \begin{bmatrix} \frac{q_a^{max}}{S} & 0 \\ 0 & \frac{q_b^{max}}{S} \\  \frac{-q_a^{max}}{S} & 0
\\ 0 & \frac{-q_b^{max}}{S}  \end{bmatrix}u,
\end{align*}
where $\tau_i=\frac{S}{a_i} \sqrt{\frac{2h_i^0}{g}}$, $i=1,\dots,4$, is the time constant for tank $i$.

Using zero-order hold method with a sampling time of $5$ seconds  we
obtain the discrete time model of type \eqref{simp1}, with the
partition $x^1 \leftarrow \left[ x^1~ x^4\right]^T$ and $x^2
\leftarrow \left[ x^2~ x^3\right]^T$. For the input constraints of
the MPC scheme we consider the practical constraints  of the ratios
of the three way valves for our plant, i.e $u^i \in [0.15,\
0.8]-\gamma^i_0$, where $\gamma^i_0$ is the linearization input. Due
to the fact that our plant has overflow sensors fitted to the tanks
and an emergency shutoff program, we do not introduce constraints
for the states. For the stage cost we have taken the weighting
matrices to be $Q^i=I_{n_i}$ and $R^i=0.01I_{m_i}$.


\subsection{Implementation of the MPC scheme using MPI}
\label{mpi} In this section we underline the benefits of Algorithm
PCDM when it is implemented in an appropriate  fashion for  the
quadruple tank MPC scheme.  We implemented for comparison, Algorithm
PCDM and that of \cite{Stewart}. Both algorithms were implemented in
C programming language, with parallelization ensured via MPI and
linear algebra operations done with CLAPACK. Algorithm
\cite{Stewart} requires solving, at each step, $2$ QP problems in
parallel, problems which cannot be solved in closed form. For
solving these QP problems, we use the \textit{qpip} routine of the
QPC toolbox \cite{QPC}. The algorithms were implemented on a PC,
with 2 Intel Xeon E5310 CPUs at 1.60 GHz and 4Gb of RAM. For the MPC
problem in this subsection we control the plant such that the levels
and inputs will reach those of the steady state linearization values
$h^0$ and $\gamma^0$.

\begin{figure}[h!]
\centering
\includegraphics[scale=0.28]{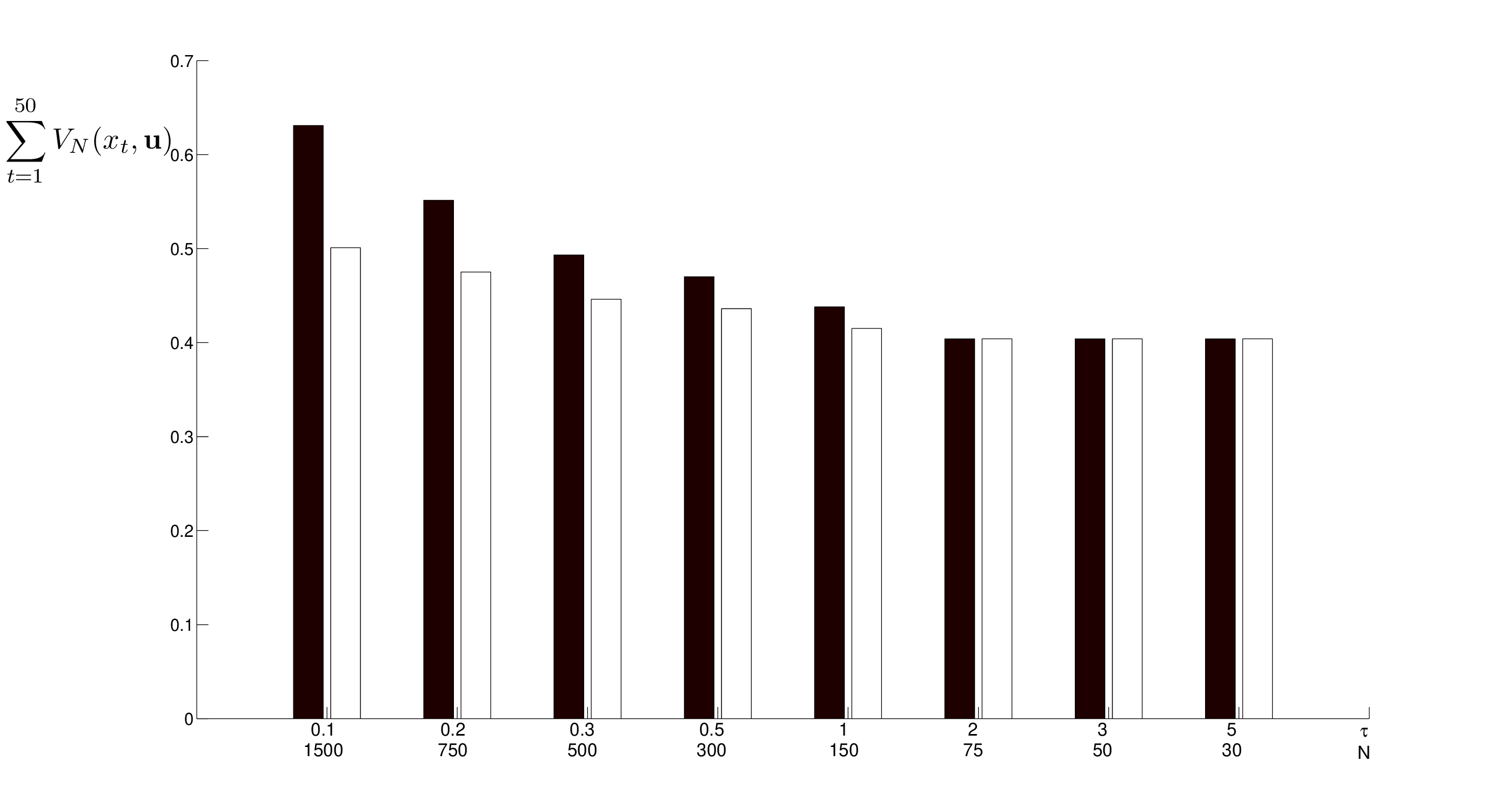}
\caption{Total costs for $50$ MPC steps with PCDM (white) and
\cite{Stewart} (black).} \label{mpc_costs}
\end{figure}

Figure \ref{mpc_costs} outlines a comparison of the two algorithms
for solving this quadruple tank MPC problem, considering a
prediction time of $150$ seconds, for different prediction horizons
$N$ and sampling time $\tau$, such that $\tau N=150$ seconds. The
bar values represent the total sum $\displaystyle \sum_{t=1}^{50}
V_N(x_t,\bo{u})$ for $50$ MPC steps, where $\bo{u}$ is calculated
either with PCDM or with the algorithm from \cite{Stewart}. For the
same $50$ simulation steps, we outline in Table \ref{mpc_iterations}
a comparison of the average number of iterations achieved by both
algorithms and the performance loss, i.e. a percentile difference
between the suboptimal cost achieved in Figure \ref{mpc_costs}
($\displaystyle \sum_{t=1}^{50} V_N(x_t,\bo{u})$) and the optimal
costs that were precalculated with Matlab's quadprog ($\displaystyle
\sum_{t=1}^{50} V_N^*(x_t)$), both for the time $\tau$ and
prediction horizon $N$. Note that our total cost is usually better
than that of \cite{Stewart} when the available time is short
($\tau<2$) and for $\tau \geq 2$ both algorithms solve the
corresponding optimization problem exactly. Also note that, due to
its low complexity iteration, our algorithm performs more than ten
times the amount of iterations than the algorithm from
\cite{Stewart}.
\begin{center}
\begin{table}[h!]
\begin{center}
\small \setlength{\tabcolsep}{2pt}
 \begin{tabular}{|c|c|c|c|} \hline
\multicolumn{2}{|c|}{} & PCDM & \cite{Stewart} \\ \hline
$\tau$ & N & Iter / Perf. Loss (\%) & Iter / Perf. Loss (\%) \\
 \hline 0.1 & 1500 & 7 / 23.9 & 1 / 60.18\\
\hline 0.2 & 750 & 30 / 17.59  & 2 / 36.48 \\
 \hline 0.3 & 500 & 240 / 10.42  & 5 / 22.1 \\
\hline 0.5 & 300 & 1803 / 7.94 & 22 / 16.36 \\
 \hline 1 & 150 & 12244 / 2.74  & 258 / 8.44 \\
 \hline 2 & 75 & 67470 / 0 & 2495 / 0 \\
 \hline 3 & 50 & 153850 / 0 & 8663 / 0 \\
\hline 5 & 30 & 382810 / 0 & 38110 / 0  \\       \hline
\end{tabular}
\caption{Number of iterations and performance loss, for different
times $\tau$.} \label{mpc_iterations}
\end{center}
\end{table}
\end{center}
\normalsize


\subsection{Implementation of the MPC scheme using  Siemens S7-1200 PLC}

Due to the limitations, in both hardware and programming language,
of the S7-1200 PLC, a proper implementation of any distributed
optimization algorithm, in the sense of distributed computations and
passing information between processes running on different cores,
cannot be undertaken on it.
 However, to illustrate the fact that our PCDM algorithm
is suitable for control devices with limited computational power and memory, we implemented
it in a centralized manner for an MPC scheme in order to control the quadruple tank plant.
We note that S7-1200 PLC is considered an entry-level PLC, with $50$ KB of
main memory, $2$ MB of load memory (mass storage) and $2$ KB of
backup memory. There are two main function blocks for the algorithm
itself, one that updates $q(x) = \bo{W}x + \bo{w}$ in the quadratic
objective function $f$ given the current levels of the four tanks
and one in which Algorithm PCDM is implemented for solving problem
\eqref{prob_princ}. Both blocks contain Structured Control Language
 which corresponds to  IEC 1131.3 standard. The remaining
function blocks are used for converting the I/O for the plant to
corresponding metric values. The elements of the problem which
occupy the most memory is the $\bo{Q} \in \rset^{2N \times 2N}$
matrix of the objective function $f$ and matrix $\bo{W} \in
\rset^{2N \times 4}$  for updating $q(x)$. Both matrices are
precomputed offline using Matlab and then stored in the work memory
using Data Blocks. The components of the problem which require
updating are the input trajectory vectors $\bo{u}^i$ and the vector
$q(x)$ of the objective function $f(\bo{u})$ which is dependent  of
the current state of the plant and of the current set point. The
evolution of the tank levels and input ratios of the plant are
recorded in Matlab on the plant's PC workstation, via an OPC server
and Ethernet connection. In accordance with the imposed sample time
of $5$ seconds, the cycle time of the S7-1200 PLC is also limited to this interval.

\begin{table}
\centering
 \begin{tabular}{|l|c|c|c|}
\hline
Cycle Time & \multicolumn{3}{|c|}{5 \text{s}} \\
\hline
  Prediction Horizon $N$ & 10 & 20 & 30 \\
\hline
Maximum Number of Iter. & 104  & 39 & 15 \\ \hline
Used Memory (\%) & 59 & 72 & 88 \\
 \hline
 \end{tabular}
\caption{Available number of iterations and memory usage of Alg. PCDM}
\label{iterations}
\end{table}

Due to this cycle time, the limited size of the
S7-1200's work memory and its processing speed,  the number of
iterations of the Algorithm PCDM that can be computed are also
limited. In Table \ref{iterations} the number of iterations
available per prediction horizon,  included in the $5$ seconds cycle
time, and the memory requirements for these prediction horizons are
presented.  Although the numbers of computed iterations seem small,
we have found in practice that the suboptimal MPC scheme still
stabilizes the quadruple tank process and ensures set point
tracking.
\begin{figure}[h!]
\centering
\begin{subfigure}{\textwidth}
\includegraphics[width=\textwidth]{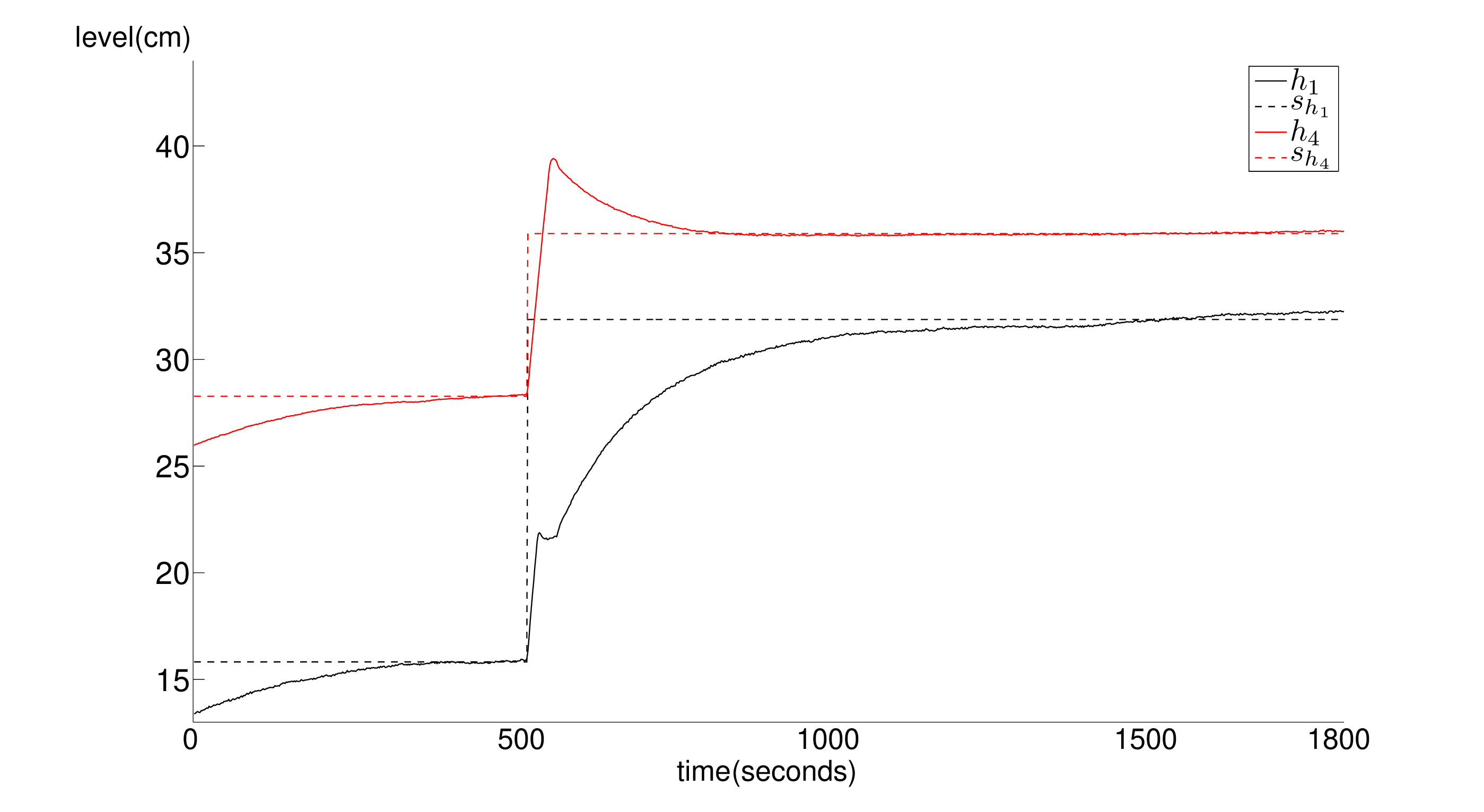}
\end{subfigure}
\begin{subfigure}{\textwidth}
\includegraphics[width=\textwidth]{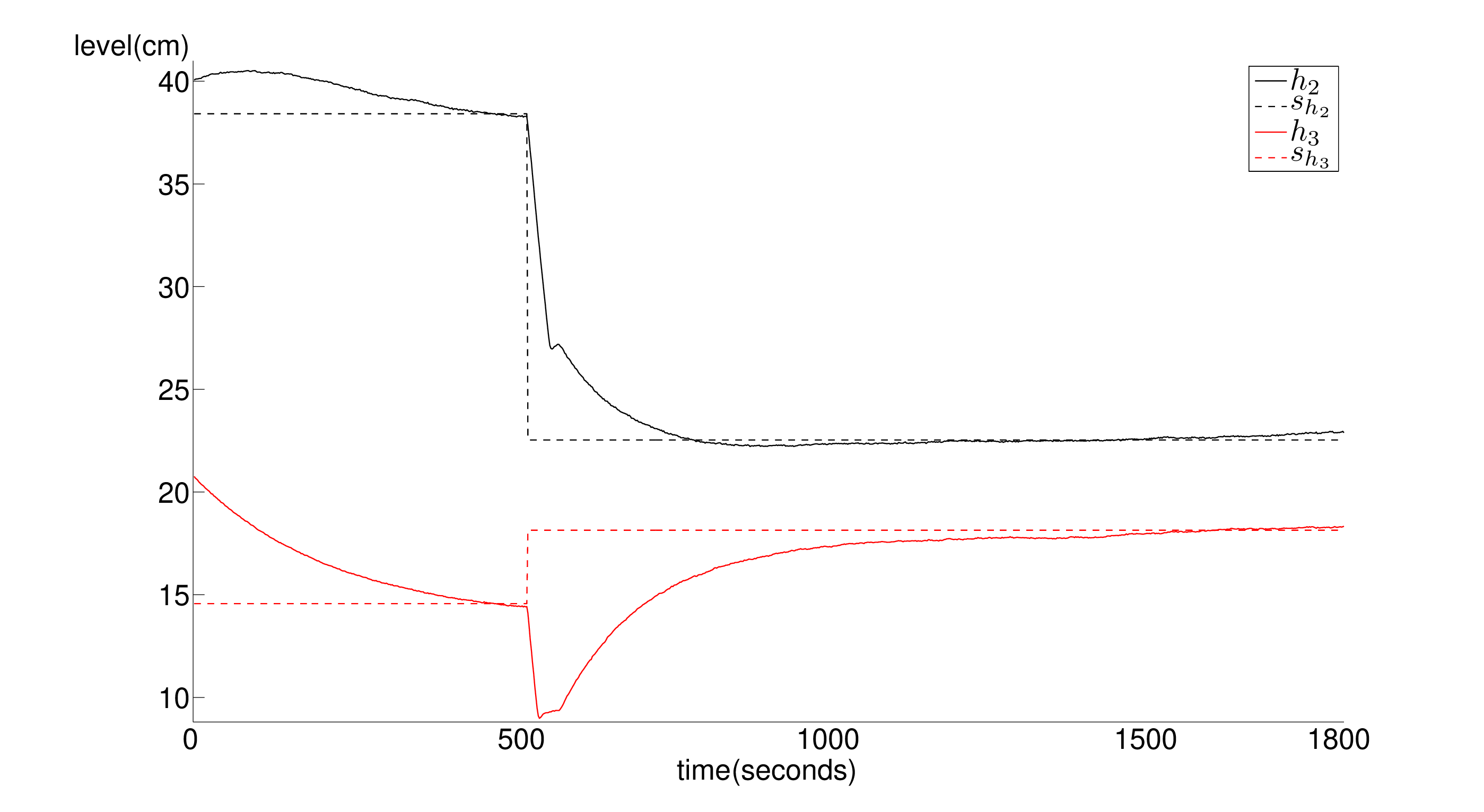}
\vspace{-1.3cm}
\end{subfigure}
\caption{Evolution of tank levels 1-4 (top), 2-3 (bottom), with continuous lines, against their respective set points, with dashed lines.} \label{4tank_results}
\end{figure}
The results of the control process are presented in Fig.
\ref{4tank_results}  for a prediction horizon $N=20$: the continuous
lines represent the evolution of water levels in each of the four
tanks, while the dashed lines are their respective set points. We
choose two set points. We first let the plant get near its first
setpoint, after which we choose a new set point which is an
equilibrium point for the plant. As it can be observed from the
figure, the MPC scheme still steers the process to the respective
set points.



\subsection{Implementation of MPC scheme for random network systems}
We now wish to outline a comparison of results between algorithm
PCDM and that of \cite{Stewart} when solving   QP problems arising
from  MPC for random network systems. Both algorithms were
implemented in the same manner as described in Section \ref{mpi}. We
considered random network systems with dynamics  \eqref{mod3}
generated as follows: the entries of system matrices $A^{ij}$ and
$B^{ij}$ are taken  from a normal distribution with zero mean and
unit variance. Matrices $A^{ij}$ are then scaled, so that they
become neutrally stable.  Matrices $Q^i \succeq 0$ and $R^i \succ 0$
are random. The input variables are constrained to lie in box sets
whose boundaries are generated randomly. The terminal cost matrices
$P^i$ are taken to be the solution of the SDP problem given in Lemma
\ref{lemma1}. For each subsystem the number  of inputs is taken $m_i
= 5$ or $m_i=10$. We let the prediction horizon  range between $N=6$
to $N=120$. The subsystems are arranged in a ring, i.e.
$\neigh=\{i-1, i, i+1\}$. We first considered $M=8$ subsystems,
matching the number of cores on our PC. Parallel implementation was
 also carried out  for $M=16$ subsystems, with each core of the PC running
two processes.  The resulting random QP problems have $p= M N m_i $
variables. The stopping criterion for each algorithm is
$f(\bo{u}_k)-f^* \leq 0.001$, with $f^*$ being precomputed for each problem using Matlab's quadprog. For each prediction horizon, $10$
simulations were run, starting from different random initial states.

\begin{table}[h]
\centering \small
\begin{tabular}{|c|c|c|c|c|c|c|c|} \hline
 \multicolumn{2}{|c|}{}  & \multicolumn{2}{|c|}{PCDM} & \multicolumn{2}{|c|}{\cite{Stewart}} & PCDM  & Quadprog \\
  \multicolumn{2}{|c|}{}  &     \multicolumn{2}{|c|}{}   &  \multicolumn{2}{|c|}{}      & centralized     &          \\ \hline
 M & p & CPU (s) & Iter & CPU (s) & Iter & CPU (s) & CPU (s)  \\ \hline
8 &  480 & \! 0.47 \! & 1396 & 1.904 & 682 & 0.663 & 1.08\\
  & 960 & 2.21 & 2839 & 21.52 & 1475 & 9.15 & 3.57 \\
  & 3200 & 256.4 & 8671 & 911.2 & 4197 & 265.3 & 39.8 \\
  & 4800 &   857.2   & 12750 &   7864.4 &  6182    &  1114       & 139.9 \\
 & 9600 & 2223.1  & 16950   &   *    &  *         & 3125   &  307.5   \\ \hline
16 & 480 & 4.36 & 2600 & 4.66 & 1615 & 0.99 & 0.97 \\
  & 960 & 15.02 & 4792 & 25.18 & 2798 & 14.17 & 3.21 \\
  & 3200 & 377.6 & 13966 & 612.8 & 8462 & 423.1 & 41.3 \\
  & 4800 & 1524.7 & 23539 & 3061.7 & 14241 & 2161.1 & 134.03 \\
  & 9600 & 3415.1  & 29057   &  *    &  *         & 4773   &  308.4  \\\hline
\end{tabular}
\caption{CPU time in seconds and nr. of iterations for alg. PCDM and
\cite{Stewart}. } \label{numerical_table}
\end{table}

\normalsize

Table \ref{numerical_table} presents the average CPU time in seconds for the
execution of each algorithm. It illustrates that Algorithm PCDM,
with its design for distributed computations and simple iterations,
usually performs better than that in \cite{Stewart}, where the
assumption is that for each iteration, a QP problem of size $ p/M$
needs to be solved.  The entries with $*$ denote that the algorithm
would have taken over $5$ hours to complete. Also note that our
implementation of the algorithm from \cite{Stewart}, for problems of
larger dimensions, i.e starting with $p=3200$,  takes less time for
it to complete if the problem is divided between $M=16$ subsystems
than $M=8$. This is due to the fact that the solver \textit{qpip}
takes much more time to solve problems of size $600$ in the case of
$p=4800$ and $M=8$ than problems of size $300$ for $p=4800$ and
$M=16$. Also, the transmission delays between subsystems are
negligible in comparison with these \textit{qpip} times.  We have
also implemented Algorithm PCDM in a centralized manner, i.e. without using MPI
and, as can be seen from the table, we gain speedups of computation
when the algorithm is parallelized. Algorithm  PCDM is outperformed
by Matlab's quadprog, but do note that quadprog is not designed for
distributed implementation and there are no transmission delays
between processes.


\section{Conclusions}
In this paper we have proposed  a parallel optimization algorithm
for solving  smooth convex  problems with separable constraints that
may arise e.g in MPC  for general linear systems comprised of
interconnected subsystems. The new optimization algorithm is based
on the block coordinate descent framework but with  very simple
iteration complexity and using  local information. We have shown
that for strongly convex objective functions it has linear
convergence rate. An MPC scheme based on this optimization algorithm
was derived, for which every subsystem in the network can compute
feasible and stabilizing control inputs using distributed
computations. An analysis for obtaining local terminal costs  from a
distributed viewpoint was made which guarantees stability of the
closed-loop interconnected system. Preliminary numerical tests show
that this algorithm is suitable for MPC applications, especially
those with hardware that has low computational power.


\end{document}